\documentclass[11pt,sumlimits,intlimits]{amsart}
\usepackage{url,cite,hyperref,amsmath,amsthm,amssymb,mathtools}
\usepackage[margin=1.25in]{geometry}
\usepackage{xcolor}
\usepackage{comment}
\usepackage[all]{xy}

\newtheorem{theorem}{Theorem}[section]

\newtheorem{corollary}[theorem]{Corollary}

\theoremstyle{definition}
\newtheorem{definition}[theorem]{Definition}

\newtheorem{example}[theorem]{Example}

\newcommand{\la}{\langle}
\newcommand{\ra}{\rangle}

\newcommand{\Z}{\mathbb{Z}}

\newcommand{\F}{\mathbb{F}}

\newcommand{\C}{\mathbb{C}}
\newcommand{\N}{\mathbb{N}}

\newcommand{\ve}{\varepsilon}

\numberwithin{equation}{section}

\title{Non-MF groups and non-finite full group C*-algebras}
\author{Caleb Eckhardt}
\address{Department of Mathematics, Miami University, Oxford, OH, 45056}
\email{eckharc@miamioh.edu}
\thanks{}
\begin{document}
\thanks{This work was partially supported by an AMS-Simons research enhancement grant for PUI faculty GR000765.}
\maketitle
\begin{abstract} Let $\Gamma$ be a property (T) group admitting an injective, non-surjective endomorphism and let $G$ be the associated ascending HNN-extension. Let $W=\left(\bigoplus_{G/\Gamma} \Z/2\Z\right) \rtimes G.$
We show that $W$ is not an MF group and that $C^*(G)$ is not a finite C*-algebra.  The ideas and proofs were generated by ChatGPT 5.6 Sol, we have only refined their arguments in a hopefully more palatable form.
\end{abstract}
\section{Acknowledgements and AI Statement} 
Scott Sauer also used ChatGPT to construct non-MF groups \cite{Sauer26}.  While waiting to see his examples, I used ChatGPT 5.6 Sol to construct examples using the non-sofic examples of Kun and Thom \cite{Kun26} that built on the OpenAI examples \cite{OpenAI26}  as a starting point. The ideas came from ChatGPT, I internalized them and then wrote this draft  so I take responsibility for any errors or inaccuracies. The examples presented here and by Scott's AI are very, very similar as are the proof ideas.  Scott is producing a Lean proof and my goal is expository.  I have no plan to publish this preprint. 
 My only goal was to write it in a manner that makes it easier for other mathematicians to  see the main idea.  
 
 In addition to Scott I would like to thank Alon Dogon, Francesco Fournier-Facio, Tatiana Shulman and Rufus Willett for helpful comments.  In particular, Rufus showed me that an earlier version of Theorem \ref{thm:nonfinite} actually provided examples of non-finite full group C*-algebras and Francesco showed me how to extend an earlier proof of Theorem \ref{thm:main} to greatly generalize the class of groups that are not MF.

\section{Introduction} 
Following \cite{Carrion13} we say that a discrete group is MF (matricially finite) if it embeds into the unitary group of the C*-algebra $\prod_{n=1}^\infty M_{d_n}/\oplus_{n=1}^\infty M_{d_n}$ where $(d_n)$ is a sequence of positive integers and $M_k$ denotes the $k\times k$ complex matrices. The term MF comes from Blackadar and Kirchberg's MF-algebras \cite{Blackadar97}--a C*-algebra is MF if it embeds into $\prod_{n=1}^\infty M_{d_n}/\oplus_{n=1}^\infty M_{d_n}.$   Another natural, and much stronger, definition of an MF-group would be that the reduced group C*-algebra is an MF C*-algebra\footnote{One may also further require that the canonical trace is well approximated see \cite[Definition 1.2]{Schafhauser26}}. Let's call this stronger property C*-MF just for this Introduction.  Clearly C*-MF implies MF, but the converse is false. Indeed, every residually finite group is MF  but showing that $C^*_r(\F_n)$ is C*-MF is a deep result \cite{Haagerup05} of Haagerup and Thorbj\o rnsen.  

There are a lot of MF groups.  It is more-or-less trivial that all LEF groups, and hence all maximally almost periodic groups are MF. Until \cite{Carrion13} I think that there were no non-trivial (i.e. non LEF) examples of MF-groups.  Carri\'{o}n, Dadarlat and I showed  \cite[Theorem 2.13]{Carrion13} that MF is closed under taking central quotients providing some easy-to-construct non-trivial  MF groups.  The breakthrough result of Tikuisis, White and Winter \cite{Tikuisis17} established that \emph{all} amenable groups are C*-MF.  Note that MF and C*-MF are actually equivalent for amenable groups (see \cite[Theorem 2.8]{Carrion13}).  We mention that in the non-amenable direction there has been significant progress on C*-MF groups, see for example \cite{Hayes15,  Louder25, Magee26, Rainone19, Schafhauser26}. 

My personal interest in these groups is the light they shed on operator algebras.
Six or so years ago a great open question was whether or not every stably finite C*-algebra is an MF C*-algebra (the converse is trivial).  The famous MIP$^*$=RE paper \cite{Ji20} implies the existence of a separable, stably finite, non-MF C*algebra (see \cite{Goldbring24} Proposition 6.1 and Remark 6.2). This technically solves the problem but it doesn't let us really get our hands on an example.  Here we provide explicit examples of  reduced group C*-algebras $C^*_r(W)$ that are (trivially) stably finite but not MF. 

Another fun open question was whether or not every full group C*-algebra $C^*(G)$ is a finite C*-algebra.  The closest thing we had was this nice observation \cite{Yamashita11} of Yamashita providing an example of a non-quasidiagonal full group C*-algebra.  As in the quasidiagonal case, it turned out to be very easy to produce examples of non-finite group C*-algebras--once you have the right groups (see Theorem \ref{thm:nonfinite}). This was a stepping stone in the eventual construction of non-MF groups. Let me provide some detail and insight into the beginning of this investigation.

Let $\Gamma$ be a non-coHopfian group (i.e. it admits an injective, non-surjective endomorphism) and let $G$ be the associated ascending HNN-extension.  Let $t\in G$ such that $H = t\Gamma t^{-1}<\Gamma.$  Then in any finite-dimensional representation $(\pi,V)$ of $G$ we have $V^H=V^\Gamma.$ The starting point of this investigation was to try and promote this idea to stably finite C*-algebras, i.e. if a representation factors through a stably finite C*-algebra do the fixed points of $H$ and $\Gamma$ have to be the same? The answer is ``no" in general,  but if we additionally add property (T) then the Kazhdan projections provide a way to encode fixed vectors algebraically in the maximal group C*-algebras, and maybe we have a chance.  It turns out this is enough to construct a non-finite group C*-algebra but not enough to construct non-MF groups. To construct the non-MF groups I pointed the machine at Kun and Thom's examples \cite{Kun26} hoping the non-soficity would rub off.  I think the argument it came up with, after refining, is ultimately elementary and clever.

Providing a proof outline requires introducing a bit of notation so I have placed the outline at the beginning of the proof.

\section{Non-finite group C*-algebras}
Recall that a unital C*-algebra $A$ is called finite if $x^*x=1$ implies $xx^*=1$ for all $x\in A.$ See \cite[Lemma 5.12]{Rordam00} for several useful reformulations of this definition. In particular if $A$ is a unital finite C*-algebra and a projection $p\in A$ is unitarily equivalent to a subprojection $q\leq p$ then we have $p=q.$ We refer the reader to \cite{Brown08} for information about group C*-algebras.

\label{sec:propfull} Let $\Gamma$ be a discrete property (T) group and $C^*(\Gamma)$ its full group C*-algebra.  The Kazhdan projection $p_\Gamma\in C^*(\Gamma)$ is the unique central projection satisfying: for each unitary representation $(\pi,\mathcal{H})$ of $\Gamma$ we have $\pi(p_\Gamma)(\mathcal{H}) = \mathcal{H}^\Gamma$, where $\mathcal{H}^\Gamma$ are the fixed vectors (see e.g. \cite[3.7.6]{Higson00}). We recall that for subgroups $H\leq \Gamma$ the natural inclusion $\C[H]\subseteq \C[\Gamma]$ extends to an inclusion $C^*(H)\subseteq C^*(\Gamma)$ since the induction of a unitary representation $\pi$ from $H$ to $\Gamma$ doesn't change the norm of an element of $\pi(\C[H]).$

\begin{theorem}\label{thm:nonfinite} Let $\Gamma$ be a non-coHopfian  property (T) group, i.e. there is an injective endomorphism $\phi:\Gamma\to \Gamma$ that is not surjective. Let $G=\la \Gamma, t | \phi(\gamma)=t\gamma t^{-1}; \gamma\in \Gamma \ra$ be the ascending HNN extension associated with $\phi.$
 Let $H=t\Gamma t^{-1}<\Gamma.$  Let $p_H,p_\Gamma\in C^*(\Gamma)$
be the Kazhdan projections.  Let $A$ be a unital finite C*-algebra and $\pi:C^*(G)\to A$ a unital *-homomorphism. Then $\pi(p_H)=\pi(p_\Gamma).$  In particular $C^*(G)$ is not a finite C*-algebra nor is $C^*(\tilde{G})$ for any $G\leq \tilde{G}.$  
\end{theorem}
\begin{proof} We think of $A\subseteq B(\mathcal{H})$ for some Hilbert space $\mathcal{H}.$
Conjugation by $t$ gives an isomorphism of $C^*(\Gamma)$ onto $C^*(H)$ that carries $p_\Gamma$ to $p_H$ by uniqueness of Kazhdan projections. Hence $\pi(t)\pi(p_\Gamma)\pi(t)^*=\pi(p_H).$  Since $H\leq \Gamma$ we have $\mathcal{H}^\Gamma \subseteq \mathcal{H}^H$ hence $\pi(p_\Gamma)\leq \pi(p_H).$  
Hence $\pi(p_H)$ is unitarily equivalent to $\pi(p_\Gamma)$ which is a subprojection of $\pi(p_H)$.  Since $A$ is finite we have $\pi(p_H)=\pi(p_\Gamma).$

Finally notice that $p_H\neq p_\Gamma.$ Indeed the left quasi-regular representation on $\ell^2(\Gamma/H)$ is a unitary representation with different sets of fixed vectors for $H$ and $\Gamma.$  Hence $C^*(G)$ is not finite. Since finiteness is a hereditary property we have $C^*(\tilde{G})$ is not finite for any $G\leq \tilde {G}.$

\end{proof}

\begin{example}\label{ex:nonfinite} Examples of non finite full group C*-algebras.  By Theorem \ref{thm:nonfinite} we only need to point to examples of non-coHopfian property (T) groups. The group $\Gamma = \Z^3 \rtimes \operatorname{SL}_3(\Z)$ is such an example (see \cite[Proposition 1.1]{Cornulier07} for this example and more).  See also
  \cite[Theorem 1.4]{Ollivier07}. 
\end{example}

\section{Non-MF groups}
\begin{definition} A discrete group $G$ is called MF if there is a sequence of natural numbers $d=(d_n)$ and a faithful homomorphism from $G$ into the unitary group of the C*-algebra 
\begin{equation*}
\mathcal{Q}_{d} = \prod_{n=1}^\infty M_{d_n} /\bigoplus_{n=1}^\infty M_{d_n}.
\end{equation*}
\end{definition}

We consider generalized wreath products similar to those in the preprint of Kun and Thom \cite{Kun26}. We recall the details below

\begin{definition}\label{def:W} Let $\Gamma$ be a property (T) group contained in $G$ and suppose $t\in G$ satisfies $t\Gamma t^{-1} <\Gamma.$   Set $H=t\Gamma t^{-1}.$ Let $A=\bigoplus_{G/\Gamma} \Z/2\Z$ and let $G$ act as permutations.  Set $W=A\rtimes G.$ For each $x\in G/\Gamma$ let $a_x\in A$ be the coordinate function that is non-trivial at $x.$

\end{definition}
\begin{theorem}\label{thm:main} Let $W$ be the group defined in Definition \ref{def:W}.  Let $\gamma \in \Gamma\setminus H.$  
Set $b_\gamma=a_{\gamma t\Gamma}a_{t\Gamma}.$  Since $\gamma t \Gamma \neq t\Gamma$ we have $b_\gamma\neq 1.$ Suppose that $\pi:W\to U(\mathcal{Q}_d)$ is a homomorphism.  Then $\pi(b_\gamma)=1.$  In particular $W$ is not MF nor is any group containing $W.$
\end{theorem}
 \begin{proof}  Now that the notation is set we outline the proof.
\newline 
Since $\mathcal{Q}_d$ is a finite C*-algebra  the image of the Kazhdan projections $\pi(p_H)=\pi(p_\Gamma)$ are equal by Theorem \ref{thm:nonfinite}. Therefore for any representation $(\rho,\mathcal{K})$ of $\mathcal{Q}_d$ we have $\mathcal{K}^H=\mathcal{K}^\Gamma$ (applied to the representation $\rho\circ\pi$).  The idea is to assume we have a representation $\pi$ with $\pi(b_\gamma)\neq 1$ and then exhibit a representation $\rho$ of $\mathcal{Q}_d$ with $\mathcal{K}^H\neq\mathcal{K}^\Gamma.$
To achieve this we replace $\pi|_G$ with coordinate-wise conjugation action in a weighted HS-norm. This provides a new representation $\sigma$ of $G$ on $\mathcal{Q}_{d'}$ (for some new sequence $d'$). Introducing the coordinate wise conjugation is done to mimic the way $G$ acts on $A$ and in particular to try and produce a vector fixed by $H$ but not $\Gamma.$  Certain difficulties arise (explained in the proof) that thwart the naive approach of building such a vector.  To circumvent this difficulty we instead construct a 1-cocycle $c$ for $\sigma|_\Gamma$ that is trivial on $H$ but not $\Gamma.$  Then property (T) enters a second time to force $c$ to be a coboundary thus producing a vector fixed by $H$ but not by $\Gamma.$

We now begin the proof.  Fix $\gamma_0\in \Gamma\setminus H$ and let $\pi:W\to U(\mathcal{Q}_d)$ be a homomorphism. Assume that $\pi(b_{\gamma_0})\neq 1.$ We use $\pi$ for the representation of $W$ and the linearized representation $\pi:C^*(W)\to \mathcal{Q}_d.$

Note that $C^*(A)\cong C(X)$ with $X$ a Cantor space.  Loring showed in \cite{Loring98} that $C(X)$ is weakly semiprojective hence there is a *-homomorphic lifting $\tilde{\pi}:C^*(A)\to \prod_{n=1}^\infty M_{d_n} $ of $\pi|_{C^*(A)}.$  It follows that there are unital maps $\pi_n:W\to U(M_{d_n})$ such that
\begin{equation}\label{eq:lifting}
\| \pi_n(gh)-\pi_n(g)\pi_n(h) \| \to 0 \text{ for all }g,h\in W \quad \text{ and }\quad \pi_n(a_xa_y)=\pi_n(a_x)\pi_n(a_y) \text{ for all }x,y \in G/\Gamma.
\end{equation}
Hence for all $g\in G$ and $x\in G/\Gamma$ we have
\begin{equation}\label{eq:conj}
 \|  \pi_n(g)\pi_n(a_x)\pi_n(g)^* - \pi_n(a_{g.x}) \|\to 0.
\end{equation}
Let $g\in G$ and $x,y\in G/\Gamma.$  Then (\ref{eq:conj}) implies that for all sufficiently large $n$, depending on $g,x,y$,  we have 
\begin{equation*}
\| \pi_n(g)(\pi_n(a_x)-\pi_n(a_y))\pi_n(g)^* - (\pi_n(a_{g.x})-\pi_n(a_{g.y})) \| <2.
\end{equation*}
Since the self-adjoint matrices $\pi_n(g)(\pi_n(a_x)-\pi_n(a_y))\pi_n(g)^*$ and  $\pi_n(a_{g.x})-\pi_n(a_{g.y})$ both have spectrum contained in $\{ -2,0,2 \}$ and they are within 1 of each other their corresponding eigenspaces have the same dimension. Therefore they are unitarily equivalent. In particular we have for all sufficiently large $n$
\begin{equation}\label{eq:invariance}
\|  \pi_n(a_x)-\pi_n(a_y) \|_{\text{HS}}  =   \| \pi_n(a_{g.x})-\pi_n(a_{g.y}) \|_{\text{HS}},
\end{equation}
 where $\| \cdot \|_{\text{HS}}$ denotes the unnormalized Hilbert-Schmidt norm.

Let $x,y\in G/\Gamma.$  Then $\pi_n(a_x)\pi_n(a_y)$ is a self-adjoint unitary.  Therefore if $\pi_n(a_x)\pi_n(a_y)\neq 1$ then we have 
\begin{equation}\label{eq:HSest}
\|  \pi_n(a_x) - \pi_n(a_y) \|_{\text{HS}} = \|  \pi_n(a_x)\pi_n(a_y) - 1 \|_{\text{HS}}\geq 2.
\end{equation}
Now set $x_0 = t\Gamma \in G/\Gamma.$  By our contradiction hypothesis there are infinitely many values of $n$ so $\pi_n(a_{\gamma_0.x_0})\pi_n(a_{x_0}) = \pi(b_{\gamma_0})\neq 1.$ Therefore by (\ref{eq:HSest}) there are infinitely many values of $n$ so  $\|  \pi_n(a_{\gamma_0.x_0}) - \pi_n(a_{x_0}) \|_{\text{HS}}\geq 2.$

By considering a subsequence of $(d_n)$ we may, without loss of generality, assume that 
\begin{equation}\label{eq:HSgap}
\|  \pi_n(a_{\gamma_0.x_0}) - \pi_n(a_{x_0}) \|_{\text{HS}}\geq 2  \text{ for all } n.
\end{equation}
We also record the following consequence of (\ref{eq:invariance}) and the triangle inequality for further use:  Let $g,h\in G$ then for sufficiently large $n$ (dependent on $g$ and $h$) we have
\begin{equation}\label{eq:triangleineq}
\|  \pi_n(a_{gh.x_0}) - \pi_n(a_{x_0}) \|_{\text{HS}} \leq \|  \pi_n(a_{g.x_0}) - \pi_n(a_{x_0}) \|_{\text{HS}} + \|  \pi_n(a_{h.x_0}) - \pi_n(a_{x_0}) \|_{\text{HS}}
\end{equation}

We now build our representation $(\sigma, \mathcal{K})$ of $G$ that factors through a finite C*-algebra with $\mathcal{K}^H\neq \mathcal{K}^\Gamma$ thus providing our contradiction. 

We pause to point out where a naive attempt at a contradiction fails.  Consider the vector $\eta=(\pi_n(a_{x_0}))$ viewed as an element of the ultraproduct in the normalized Hilbert-Schmidt norm. View $\pi(G)$ as acting via conjugation.  Then $\eta$ is clearly fixed by $H$ and is a natural choice for an element not fixed by $\gamma_0.$  Indeed in the \emph{operator norm} this is precisely what happens. But there is no reason for this to work in the normalized HS-norm.  The following is the AI's workaround to this roadblock.

Since $\Gamma$ has property (T) it is finitely generated.  Choose $\gamma_1,...,\gamma_m\in \Gamma\setminus H$ such that
\begin{equation*}
\Gamma = \la H, \gamma_0,...,\gamma_m \ra.
\end{equation*}
For each $n$ set
\begin{equation*}
k_n:=\max_{i=0,...,m}\|  \pi_n(a_{\gamma_i.x_0}) - \pi_n(a_{x_0}) \|_{\text{HS}} \geq 2, \quad \text{ by (\ref{eq:HSgap})}.
\end{equation*}
Notice that by (\ref{eq:invariance}) we have for sufficiently large $n$ and all $i=0,...,m$
\begin{equation*}
\|  \pi_n(a_{\gamma_i^{-1}.x_0}) - \pi_n(a_{x_0}) \|_{\text{HS}} = \|  \pi_n(a_{\gamma_i.x_0}) - \pi_n(a_{x_0}) \|_{\text{HS}}
\end{equation*}

Let $g\in \Gamma.$  Then there are $h_0,...,h_r\in H$ and indices $i_1,...,i_r$ and $\ve_1,...,\ve_r\in \{  -1,1 \}$ so
\begin{equation}\label{eq:gproduct}
g = h_0\gamma_{i_1}^{\ve_1}h_1\gamma_{i_2}^{\ve_2}\cdots \gamma_{i_r}^{\ve_r}h_r.
\end{equation}
For all $h\in H$ we have $h.x_0 = ht\Gamma = t\Gamma=x_0$ hence for all $n$ we have $\pi_n(a_{x_0}) = \pi_n(a_{h.x_0}).$ We combine this observation with the above product formula for $g$ and repeated use of the triangle inequality estimate (\ref{eq:triangleineq}) to obtain for sufficiently large $n$ (here $n$ depends on $g$ and the chosen product formula for $g$)
\begin{equation}\label{eq:HSnormaction}
\| \pi_n(a_{g.x_0}) - \pi_n(a_{x_0}) \|_{\text{HS}} \leq rk_n.
\end{equation}
For each $n$ define $\mathcal{K}_n$ to be $M_{d_n}$ with rescaled HS norm
\begin{equation*}
\| T \|_{\mathcal{K}_n} = \frac{\| T \|_{\text{HS}}}{k_n}.
\end{equation*}
For each $g\in G$ define $\sigma_n(g) = \operatorname{Ad} \pi_n(g) \in U(\mathcal{K}_n).$ Note that for any $U,V\in U(d_n)$ and $T\in \mathcal{K}_n$ we have
\begin{align*}
\|  UTU^*-VTV^*\|_{\mathcal{K}_n} &\leq \| (U-V)TU^* \|_{\mathcal{K}_n} + \| VT(U^*-V^*) \|_{\mathcal{K}_n}\\
& \leq 2\| U-V \|\| T \|_{\mathcal{K}_n}.
\end{align*}
It follows that since $(\pi_n, \C^{d_n})_{n=1}^\infty$ is an approximate representation of $G$ that $(\sigma_n,\mathcal{K}_n)_{n=1}^\infty$ is also an approximate representation of $G.$  We then define $\sigma:= (\sigma_n):G\to \prod B(\mathcal{K}_n)/\oplus B(\mathcal{K}_n).$

Let $\omega$ be a non-principal ultrafilter on $\N$ and let $\mathcal{K} = \prod_\omega \mathcal{K}_n$ be the Hilbert space ultraproduct.   
Define $\rho:\prod B(\mathcal{K}_n)\to B(\mathcal{K})$ by
\begin{equation*}
\rho((T_n)_{n=1}^\infty)([\eta_n]) =  [T_n(\eta_n)].
\end{equation*}
If $\lim_{n\to \infty}\| T_n \|=0$ then clearly $\rho( (T_n)_{n=1}^\infty ) = 0$ so $\rho$ descends to a well defined homomorphism from $\prod B(\mathcal{K}_n)/\oplus B(\mathcal{K}_n)$ to $B(\mathcal{K}).$ It is a linear algebra exercise to show that $\prod B(\mathcal{K}_n)/\oplus B(\mathcal{K}_n)$ is a finite C*-algebra  hence the representation $\rho\circ\sigma:G\to B(\mathcal{K})$ factors through a finite C*-algebra. 

Next we build a cocycle $c:\Gamma\to \mathcal{K}$ for the representation $\rho\circ \sigma|_\Gamma.$  For each $n$ and $g\in\Gamma$ define $c_n(g) = \pi_n(a_{g.x_0})-\pi_n(a_{x_0})\in \mathcal{K}_n.$

Let $g\in \Gamma$ and assume it may be written as in (\ref{eq:gproduct}). By (\ref{eq:HSnormaction}) we have 
\newline
$\limsup_{n\to\infty} \| c_n(g)  \|_{\mathcal{K}_n} \leq r$ so  the following map $c:\Gamma\to \mathcal{K}$ is well defined
\begin{equation}
c(g) = [ c_n(g) ].
\end{equation}
Next we show $c$ satisfies the cocycle identity.  Let $h,g\in \Gamma.$ Since $\pi_n$ are approximately norm multiplicative we have
\begin{equation}\label{eq:opnormcocycle}
\| c_n(hg)-(c_n(h)+\sigma_n(h)c_n(g))  \| \to 0.
\end{equation}
We  show the same holds when operator norm is replaced with $\| \cdot \|_{\mathcal{K}_n}.$ Again assume $g$ is decomposed as in (\ref{eq:gproduct}). 
Let $u,v$ be commuting self-adjoint unitary matrices.  Then we have $|u-v|=\frac{1}{2}(u-v)^2.$  Hence $\operatorname{Tr}|u-v|=\frac{1}{2}\| u-v \|_{\text{HS}}^2.$ Since $c_n(hg)-c_n(h)$ is  a difference of commuting self-adjoint matrices we apply this calculation, (\ref{eq:invariance}) and  (\ref{eq:HSnormaction}) to obtain for sufficiently large $n$ (depending on $h,g$ and the product formula for $g$)
\begin{equation*}
\operatorname{Tr}| c_n(hg)-c_n(h) | \leq \frac{1}{2}\|  c_n(hg)-c_n(h) \|_{\text{HS}}^2 \leq \frac{1}{2}r^2k_n^2 
\end{equation*}
Similarly, $\operatorname{Tr}|\sigma_n(h)c_n(g)|\leq \frac{1}{2}r^2k_n^2.$ Set $D_n(h,g)=c_n(hg)-(c_n(h)+\sigma_n(h)c_n(g)).$ Then
 \begin{align*}
 \| D_n(h,g)  \|^2_{\mathcal{K}_n} & = \frac{1}{k_n^2}\operatorname{Tr}|D_n(h,g)  |^2\\
 &\leq \frac{1}{k_n^2}\left(\| D_n(h,g) \| \operatorname{Tr}|D_n(h,g)|\right)\\
 &\leq \frac{1}{k_n^2}\| D_n(h,g)\| \left(  \operatorname{Tr}|c_n(hg)-c_n(h))| + \operatorname{Tr}|\sigma_n(h)c_n(g)| \right)\\
 &\leq \frac{1}{k_n^2}\| D_n(h,g) \|r^2k_n^2 \to 0, \hspace{0.3in} \text{ by  (\ref{eq:opnormcocycle}) }
 \end{align*}
Therefore $c$ is a 1-cocycle.  Notice that for every $h\in H$ we have $c(h)=0$ because $h.x_0=x_0.$
We now show $c(\gamma_i)\neq0$ for some $i=0,1,...,m.$  For each $i=0,...,m$ consider the set
\begin{equation*}
Y_i = \{ n\in\N: k_n= \| \pi_n(a_{\gamma_i.x_0})-\pi_n(a_{x_0}) \|_{\text{HS}}\}.
\end{equation*}
Since $Y_0\cup \cdots \cup Y_m =\N$ there is an index $i$ so $Y_i\in \omega.$  It follows that $c(\gamma_i)\in \mathcal{K}$ has norm 1.

Since $\Gamma$ has property (T) by the Delorme-Guichardet theorem (see \cite[Theorem 2.12.4]{Bekka08}) there is a vector $\eta\in \mathcal{K}$ such that $c(g) = \rho\circ \sigma(g)\eta-\eta.$  Then $\eta\in \mathcal{K}^H$ but $\rho\circ\sigma(\gamma_i)\eta\neq\eta,$ a contradiction.  
 \end{proof}

\begin{corollary}\label{cor:sfnotMF} Let $W$ be as in Theorem \ref{thm:main}.  Then $C^*_r(W)$ is stably finite but not an MF C*-algebra. The same holds for any group containing $W.$
\end{corollary}
\begin{proof} All reduced C*-algebras of discrete groups admit a faithful trace and are hence stably finite.  If $C^*_r(W)$ were an MF C*-algebra, then $W$ would be an MF group, since $W$ embeds canonically into the unitary group of $C^*_r(W).$
\end{proof}

\begin{example} Examples of non-MF groups and stably finite, non-MF C*-algebras.  As in Example \ref{ex:nonfinite} there are many examples of groups $\Gamma$ satisfying the hypotheses of Theorem \ref{thm:main} and therefore of Corollary \ref{cor:sfnotMF}
 \end{example}

\end{document}